\documentclass{amsart}

\usepackage{amssymb,amsmath,enumerate}

\usepackage[shortlabels]{enumitem}
\usepackage[colorlinks]{hyperref}
\usepackage{colonequals}
\usepackage{eucal}

\usepackage{tikz-cd}
\usetikzlibrary{decorations.pathmorphing}

\newcommand{\lra}{\longrightarrow}
\newcommand{\thra}{\twoheadrightarrow}
\newcommand{\xla}{\xleftarrow}
\newcommand{\xra}{\xrightarrow}
\newcommand{\ges}{\geqslant}

\newcommand{\ann}[2]{\operatorname{ann}_{#1}(#2)}
\newcommand{\builds}[1][]{\vdash_{#1}\!}
\newcommand{\fbuilds}[1][]{\models_{#1}\!}
\newcommand{\cat}[1]{{\mathsf{#1}}}
\newcommand{\cone}{\operatorname{Cone}}

\newcommand{\dcat}[1]{\cat{D}(#1)}
\newcommand{\dbcat}[2][]{\cat{D}^{\operatorname{fg}}_{\scriptscriptstyle{#1}}(#2)}
\newcommand{\dcatfl}[1]{\cat{D}^{\operatorname{fl}}(#1)}
\newcommand{\sdbcat}[1]{\cat{D}^{\operatorname{b}}(\operatorname{coh}\, #1)}

\newcommand{\env}[2]{{#2}_{#1}^{\operatorname{e}}}
\newcommand{\Ext}[4][{*}]{\operatorname{Ext}^{#1}_{#2}(#3,#4)}
\newcommand{\fdim}{\operatorname{flat\,dim}}

\newcommand{\hh}[1]{\operatorname{H}_{*}(#1)}
\newcommand{\HH}[2]{\operatorname{H}_{#1}(#2)}
\newcommand{\HC}[3][*]{\operatorname{HH}^{#1}(#3/#2)}
\newcommand{\HCM}[4][*]{\operatorname{HH}^{#1}(#3/#2;#4)}
\newcommand{\Image}{\operatorname{Im}}
\newcommand{\Ker}{\operatorname{Ker}}
\newcommand{\Hom}{\operatorname{Hom}}
\newcommand{\kos}[2]{{#1}/\!\!/{#2}}
\newcommand{\lotimes}{\otimes^{\operatorname{L}}}
\newcommand{\perf}[1]{\cat{Perf}(#1)}
\newcommand{\pdim}{\operatorname{proj\,dim}}

\newcommand{\RHom}[3]{\operatorname{RHom}_{#1}(#2,#3)}
\newcommand{\shift}{{\sf\Sigma}}
\newcommand{\spa}[1]{\mathcal{#1}}
\newcommand{\spec}{\operatorname{Spec}}
\newcommand{\Sym}{\operatorname{Sym}}

\newcommand{\supp}[2]{{\operatorname{supp}_{#1}{#2}}}
\newcommand{\thick}[1]{\operatorname{thick}(#1)}
\newcommand{\Tor}[4][{*}]{\operatorname{Tor}_{#1}^{#2}(#3,#4)}

\newcommand{\fa}{{\mathfrak a}}
\newcommand{\fm}{{\mathfrak m}}
\newcommand{\fp}{{\mathfrak p}}
\newcommand{\fq}{{\mathfrak q}}
\newcommand{\fn}{{\mathfrak n}}

\newcommand{\ve}{{\varepsilon}}
\newcommand{\vf}{{\varphi}}

\newcommand{\bsx}{{\boldsymbol x}}

\newcommand{\mcF}{\mathcal{F}}
\newcommand{\mcG}{\mathcal{G}}

\newcommand{\mcN}{\mathcal{N}}
\newcommand{\mcP}{\mathcal{P}}

\newcommand{\wt}{\widetilde}

\newcounter{intro}
\newtheorem{introthm}[intro]{Theorem}

\newtheorem{theorem}[subsection]{Theorem}
\newtheorem{proposition}[subsection]{Proposition}
\newtheorem{lemma}[subsection]{Lemma}
\newtheorem{corollary}[subsection]{Corollary}

\theoremstyle{definition}
\newtheorem{example}[subsection]{Example}
\newtheorem{chunk}[subsection]{ }

\theoremstyle{remark}
\newtheorem{remark}[subsection]{Remark}

\newtheorem*{claim*}{Claim}
\newtheorem*{ack}{Acknowledgements}

\numberwithin{equation}{subsection}

\begin{document}

\title[Locally complete intersections]{Locally complete intersection maps and\\ the proxy small property}

\author[Briggs]{Benjamin Briggs}
\address{Department of Mathematics,
University of Utah, Salt Lake City, UT 84112, U.S.A.}
\email{briggs@math.utah.edu}

\author[Iyengar]{Srikanth B.~Iyengar}
\address{Department of Mathematics,
University of Utah, Salt Lake City, UT 84112, U.S.A.}
\email{iyengar@math.utah.edu}

\author[Letz]{Janina C.~Letz}
\address{Fakult\"at f\"ur Mathematik\\ 
Universit\"at Bielefeld\\ 
33501 Bielefeld\\ 
Germany.}
\email{jletz@math.uni-bielefeld.de}

\author[Pollitz]{Josh Pollitz}
\address{Department of Mathematics,
University of Utah, Salt Lake City, UT 84112, U.S.A.}
\email{pollitz@math.utah.edu}

\thanks{Partly supported by NSF grants DMS-1700985 (SBI and JCL) and DMS-1840190 (JP)}

\date{4th February 2021}


\begin{abstract} 
It is proved that a map $\vf\colon R\to S$ of commutative noetherian rings that is essentially of finite type and flat is locally complete intersection if and only if $S$ is proxy small as a bimodule. This means that the thick subcategory generated by $S$ as a module over the enveloping algebra $S\otimes_RS$ contains a perfect complex supported fully on the diagonal ideal. This is in the spirit of the classical result that $\vf$ is smooth if and only if $S$ is small as a bimodule, that is to say, it is itself equivalent to a perfect complex. The geometric analogue, dealing with maps between schemes, is also established. Applications include simpler proofs of factorization theorems for locally complete intersection maps. 
\end{abstract}

\maketitle

\section*{Introduction}
\label{sec:intro}
This work concerns the locally complete intersection property for maps between commutative noetherian rings. While there are numerous characterizations of this property, see \cite{Avramov:2010}, none are in terms purely of the structure of the derived category as a triangulated category. Our main results, Theorems~\ref{ithm:eft} and \ref{ithm:local}, supply such characterizations. To set the stage for the discussion, let $\vf\colon R\to S$ be a homomorphism of commutative noetherian rings that is flat and essentially of finite type; the latter condition means that $S$ is a localization of a finitely generated $R$-algebra. 

We establish criteria for detecting when $\vf$ is locally complete intersection, analogous to the following criterion for smoothness: $\vf$ is smooth if and only if $S$ is perfect when viewed as a complex over the enveloping algebra $\env RS\colonequals S\otimes_RS$ via the multiplication map $\env RS\thra S$; see \cite[Proposition~(17.7.4)]{Grothendieck:1967} and \cite[Theorem~1]{Rodicio:1990}. The condition that $S$ is perfect is equivalent to the condition that $S$ is isomorphic in $\dcat{\env RS}$, the derived category of $\env RS$, to a bounded complex of finitely generated projective $\env RS$-modules. We prove:

\begin{introthm}
\label{ithm:eft}
Let $\vf\colon R\to S$ be a homomorphism of commutative noetherian rings, flat and essentially of finite type. Then $\vf$ is locally complete intersection if and only if the thick subcategory of $\dcat{\env RS}$ generated by $S$ contains a perfect complex whose support equals that of $S$.
\end{introthm}

By the \emph{support} of a complex $W$ in $\dcat{\env RS}$ we mean the set of prime ideals $\fq$ in $\spec(\env RS)$ such that ${\hh W}_\fq\ne 0$. A complex (in the derived category of some ring) is \emph{proxy small} if the thick subcategory it generates contains a perfect complex with the same support; see Section~\ref{sec:ps}. Thus Theorem~\ref{ithm:eft} can be rephrased as: $\vf$ is locally complete intersection if and only if $S$ is proxy small in $\dcat{\env RS}$.

There are other reformulations possible. Indeed, it follows from a result of Hopkins~\cite{Hopkins:1987} that if the thick category generated by $S$ contains a perfect complex with support equal to that of $S$, then it has to contain every perfect complex whose support is contained in that of $S$; see~\ref{ch:HN}. So Theorem~\ref{ithm:eft} is equivalent to the statement $\vf$ is locally complete intersection if and only if $S$ generates the Koszul complex on a finite generating set for the kernel of the multiplication map.

Theorem~\ref{ithm:eft} is a consequence of Theorem~\ref{th:global-eft} that applies to maps of finite flat dimension which is the natural context for the locally complete intersection property. That result is in turn deduced from Theorem~\ref{th:ci-local} concerning surjective homomorphisms; the latter brings out another feature of complete intersections:

\begin{introthm}
\label{ithm:local}
Let $\vf\colon R\to S$ be a surjective homomorphism of finite flat dimension. Then $\vf$ is locally complete intersection if and only if any $S$-complex that is proxy small as an $R$-complex is also proxy small over $S$.
\end{introthm}

In other words, $\vf$ is complete intersection if and only if proxy smallness \emph{ascends} along $\vf$. The forward implication that proxy smallness ascends for complete intersection maps is the content of \cite[Theorem~9.1]{Dwyer/Greenlees/Iyengar:2006b}, so the result above provides a converse. For this direction, it suffices to test that ascent holds for complexes with finite length homology. When $S$ is local one can even specify a finite collection of $S$-complexes of finite length homology whose proxy smallness detects the complete intersection property for $\vf$; see Theorem~\ref{th:lci-test-complexes}.

The notion of proxy small complexes was introduced in \cite{Dwyer/Greenlees/Iyengar:2006a} as a tool in duality theory. It has since become clear that this concept captures also interesting geometric properties of maps, and that this sheds a new light on factorization theorems. For example, in Section~\ref{sec:applications} we use Theorem~\ref{ithm:local}, more precisely, Theorem~\ref{th:ci-local}, to give simple proofs of some fundamental results concerning the factorization of locally complete intersection maps, first established by Avramov~\cite{Avramov:1999} as a consequence of his solution of a conjecture of Quillen concerning cotangent complexes. 

The statement of Theorem~\ref{ithm:local} and its proof are inspired by a result of the fourth author~\cite[Theorem~5.2]{Pollitz:2019} characterizing local rings that are complete intersection in terms of proxy smallness of complexes, thereby settling a question raised in \cite[Question~9.10]{Dwyer/Greenlees/Iyengar:2006b}. A key new ingredient in our proof is the use of  Hochschild cohomology and its action on derived categories. Recent work~\cite{Letz:2020} of the third author is also critical for it allows us to deduce global statements from local ones. Indeed,   the third author~\cite[Theorem~5.9]{Letz:2020} used \cite[Theorem~5.2]{Pollitz:2019} directly to establish Theorem~\ref{ithm:eft} when $R$ is a field.

Theorems~\ref{ithm:eft} and \ref{ithm:local} extend to morphisms of schemes, but the appropriate notion of proxy smallness involves tensor-generation. This is explained in Section~\ref{se:schemes}. Keeping in mind that geometrically the multiplication map is the diagonal embedding, Theorem~\ref{ithm:eft} readily yields the following result.

\begin{introthm}
\label{ith:lci-spa-maps}
Let $f\colon Y\to X$ be a flat, essentially of finite type, separated, morphism of noetherian schemes, and $\delta \colon Y \to {Y}\times_{X} {Y}$ the diagonal embedding. Then $f$ is locally complete intersection if and only if the thick $\otimes$-ideal of $\dcat{{Y}\times_{X} {Y}}$ generated by $\delta_*{\spa O}_Y$ contains a perfect complex whose support is the diagonal.
\end{introthm}

As with Theorem~\ref{ithm:eft}, but this time using a result of Thomason~\cite[Theorem~4.1]{Thomason:1997}, one can reformulate the theorem above to say that $f$ is locally complete intersection if and only if $\delta_*{\spa O}_Y$ tensor-generates a Koszul complex whose support is the diagonal.

The derived category of a commutative ring has been a valuable source of inspiration for results, if not also their proofs, in other tensor triangulated categories, like the category of spectra, or the stable category of modular representations of finite groups, and also in triangulated categories arising in non-commutative geometry. Theorems~\ref{ithm:eft} and \ref{ithm:local} open the way to exploring notions of complete intersection rings and maps in these categories; see Remark~\ref{re:general-lci}.

\begin{ack}
It is our pleasure to thank  Jian Liu and the referees for their comments and suggestions on an earlier version of this manuscript.
\end{ack}

\section{Proxy smallness}
\label{sec:ps}
This section is mainly a collection of definitions and  observations concerning proxy small objects in derived categories of rings and dg (= differential graded) algebras. Although the main results involve only rings, their proofs exploit dg structures extensively. We take \cite{Avramov:2010, Avramov/Iyengar/Lipman/Nayak:2010} as our basic references on this topic. By default the grading will be lower, so differentials decrease degree.

Throughout $A$ will be a dg algebra concentrated in non-negative degrees. Given a dg $A$-module $M$ we view its homology $\hh M$ as a graded $\hh A$-module. When we speak of elements in a graded object  only  homogeneous elements are considered.

A dg algebra $A$ is equipped with an augmentation map $A\lra \HH 0A$. It is a map of dg algebras where $\HH 0A$, like any ring, is viewed as a dg algebra concentrated in degree zero with zero differential. Through this map any dg $\HH 0A$-module (that is to say, a complex of $\HH 0A$-modules) inherits a structure of a dg $A$-module.

We write $\dcat{A}$ for the derived category of left dg $A$-modules, with its canonical structure as a triangulated category, equipped with suspension functor $\shift$. 

\begin{chunk}
A \emph{thick} subcategory of $\dcat{A}$ is a triangulated subcategory closed under retracts. As the intersection of thick subcategories is again a thick subcategory, for each object $M$ of $\dcat{A}$ there exists a smallest thick subcategory, with respect to inclusion, containing $M$; we denote it $\thick M$. See \cite[\S2]{Avramov/Buchweitz/Iyengar/Miller:2010a} for a constructive description of this category. Following \cite{Dwyer/Greenlees/Iyengar:2006a,Dwyer/Greenlees/Iyengar:2006b} we say that a dg module $N$ is \emph{finitely built} from $M$, or that $M$ \emph{finitely builds} $N$, if $N$ is in $\thick M$. This situation is indicated by writing $M\fbuilds[A] N$; we drop the $A$ from the notation if the ambient category is unambiguous.

A \emph{localizing} subcategory of $\dcat{A}$ is a triangulated subcategory closed under arbitrary coproducts; such a subcategory is thick. Once again mimicking \cite{Dwyer/Greenlees/Iyengar:2006a,Dwyer/Greenlees/Iyengar:2006b}, we write $M\builds[A] N$ to indicate that $N$ is in the localizing subcategory generated by $M$, and say $M$ \emph{builds} $N$, or that $N$ is \emph{built} by $M$.

It is straightforward to verify that the relations $\builds$ and $\fbuilds$ are transitive; this will be used without further mention. Evidently if $M\fbuilds N$ then $M\builds N$; the converse does not hold for arbitrary pairs of dg $A$-modules.
\end{chunk}

\begin{chunk}
\label{ch:perfectbuilding} 
All objects in $\dcat{A}$ are built from $A$; in symbols: $A\builds M$ for any $M$ in $\dcat{A}$. This is a restatement of the fact that every dg module has a semifree resolution; see~\cite[Chapter~6]{Felix/Halerpin/Thomas:2001}. It has long been known that the objects that are finitely built from $A$, that is to say the perfect dg modules, are precisely the `small objects' of $\dcat{A}$; see, for example, \cite[Theorem~2.1.3]{Hovey/Palmieri/Strickland:1997}. Recall that a dg $A$-module $M$ is \emph{small} (or \emph{compact}) provided $\Hom_{\dcat A}(M,-)$ commutes with arbitrary direct sums. When $M$ and $N$ are small dg $A$-modules
\begin{equation}
\label{eq:compact}
 M\builds N\quad \text{implies}\quad M\fbuilds N\,.
\end{equation}
For a proof see for example \cite[Lemma~2.3]{Neeman:1992a} and also \cite[Corollary~3.14]{Rouquier:2008}.
\end{chunk}

\begin{chunk}
\label{ch:rest}
Let $\vf \colon A\to B$ be a morphism of dg algebras. If $M\builds[B] N$ or $M \fbuilds[B] N$ then $M \builds[A] N$ or $M \fbuilds[A] N$, respectively, viewing $M$ and $N$ as dg $A$-modules by restricting scalars along $\vf$. 
\end{chunk}

\begin{chunk}
\label{ch:proxy}
As in \cite{Dwyer/Greenlees/Iyengar:2006a}, a dg $A$-module $M$ is \emph{proxy small} if there exists a small dg $A$-module $K$ such that $M\fbuilds K$ and $K\builds M$. We say that $K$ is a \emph{small proxy} for $M$. Evidently small objects are proxy small. When $A$ is a commutative noetherian ring, some of these conditions can be expressed in terms of support; see \ref{co:proxy-support}.
\end{chunk}

The following definition is central to this work. 

\begin{chunk}
\label{ch:ascent-defn}
Let $\vf\colon A\to B$ be a morphism of dg algebras. We say that \emph{proxy smallness ascends along $\vf$} if each dg $B$-module that is proxy small in $\dcat A$ is proxy small in $\dcat B$. The phrase \emph{proxy smallness descends along $\vf$} means that each proxy small dg $B$-module is also proxy small in $\dcat A$. Often the focus will be on ascent (or descent) of proxy smallness for dg modules in a subcategory $\cat C$ of $\dcat B$, and then we speak of proxy smallness of objects in $\cat C$ ascending/descending along $\vf$. 

For example, proxy smallness ascends and descends along $\vf$ if it is a quasi-isomorphism, for then the base change functor $F\colonequals B\lotimes_A - \colon \dcat A\to \dcat B$ is an exact equivalence of categories with quasi-inverse the restriction functor. 
\end{chunk}

From the definition it is clear that whether an object in $\dcat A$ is small or not depends only the structure of $\dcat A$ as a triangulated category. The following remark is obvious, but also obviously useful.

\begin{chunk}
\label{ch:proxy-tensor-closed}
Let $A$ and $B$ be dg algebras and $F \colon \dcat{A} \to \dcat{B}$ an exact functor preserving coproducts and small objects. If $M$ is proxy small in $\dcat{A}$ with small proxy $K$, then $F(M)$ is proxy small in $\dcat{B}$ with small proxy $F(K)$. The converse holds if $F$ is an exact equivalence.
\end{chunk}

\begin{lemma}
\label{le:h0}
Let $A$ be a dg algebra. The following statements hold.
\begin{enumerate}[\rm(1)]
\item\label{le:h0part1} The dg $A$-module $\HH 0A$ builds any $M$ in $\dcat A$ with $\HH iM=0$ for $|i|\gg 0$.
\item\label{le:h0part2} If $\HH iA=0$ for $i\gg 0$, then $\HH{0}{A}$ is proxy small if and only if $\HH{0}{A}\fbuilds A$. 
\end{enumerate}
\end{lemma}

\begin{proof}
Set $B\colonequals\HH{0}{A}$ and let $\ve\colon A\to B$ be the augmentation map.

\ref{le:h0part1} We verify this claim by an induction on the number of nonzero homology modules in $M$. Set $i\colonequals \inf \hh M$; we may assume this is finite, else $M\simeq 0$. In $\dcat A$ soft truncation yields an exact triangle 
\[
N \lra M \lra \shift^i \HH iM \lra
\]
where the induced map $\HH nN \to \HH nM$ is bijective for $n\ne i$ and $\HH iN=0$. The dg $A$-module structure on $\HH iM$ is the one induced via $\ve$. Now $B\builds \HH iM$ in $\dcat B$ and hence also in $\dcat A$. Since $N$ has one fewer nonzero homology modules than $M$ the induction hypothesis yields that $B\builds[A] N$. The exact triangle above then implies that $B\builds[A] M$.
 
\ref{le:h0part2} The non-trivial implication is that when $B$ is proxy small it finitely builds $A$. Suppose $K$ is a small proxy for $B$; in particular, $K$ builds $B$. By part \ref{le:h0part1} and the boundedness hypothesis, $B$ builds $A$, and it follows that so does $K$. However $K$ and $A$ are both small objects so $K$ finitely builds $A$; see \eqref{eq:compact}. As $B\fbuilds K$, transitivity implies once again that $B\fbuilds A$.
\end{proof}

\begin{example}
The preceding result does not extend to dg algebras $A$ with $\HH iA\ne 0 $ for infinitely many $i$. For example, if $R$ is any commutative ring and $A\colonequals R[x]$, viewed as a dg algebra with $|x|\ge 1$ and zero differential, then $R=\HH 0A$ is small in $\dcat A$, and so proxy small. However it does not build $A$, let alone finitely.
\end{example}

\begin{proposition}
\label{pr:local-to-global}
Let $R$ be a commutative noetherian ring, $A$ a dg $R$-algebra, and let $M,N$ be dg $A$-modules. Assume the $R$-modules $\hh A$, $\hh M$, and $\hh N$ are finitely generated. Then $M\fbuilds[A]N$ if and only if $M_\fp\fbuilds[A_\fp]N_\fp$ for each $\fp$ in $\spec R$.
\end{proposition}

\begin{proof}
One way to verify this is to mimic the argument for \cite[Theorem~3.6]{Letz:2020} to get the desired result. Another is to invoke the local-to-global principle~\cite[Theorem~5.10]{Benson/Iyengar/Krause:2015}, for the triangulated category consisting of dg $A$-modules with homology finitely generated over $R$, viewed as an $R$-linear category.
\end{proof}

In this work the focus is on proxy small objects in $\dcat R$, the derived category of a commutative noetherian ring $R$. Next we recollect some results specific to this context. We write $\dbcat R$ for the subcategory of $\dcat R$ consisting of $R$-complexes $M$ for which the $R$-module $\hh M$ is finitely generated. Similarly we write $\dcatfl R$ for the subcategory of $R$-complexes $M$ for which $\hh M$ has finite length. 

\begin{chunk}
\label{ch:supp}
Let $R$ be a commutative noetherian ring. The \emph{support} of an object $M$ in $\dcat R$ is the subset of $\spec R$ given by
\[
\supp RM\colonequals \{\fp\in\spec R\mid k(\fp)\lotimes_R M\not\simeq 0\}
\]
where $k(\fp)$ is the residue field $R_\fp/\fp R_\fp$ at $\fp$. For $M\in \dbcat R$ one has
\[
\supp RM = \{\fp\in \spec R\mid {\hh M}_\fp \ne 0\} = V(\ann R{\hh M})
\]
and hence it is a closed subset of $\spec R$. For example, if $K$ is the Koszul complex on a finite generating set for an ideal $I$, then $\supp RK=V(I)$.

For any subset $U$ of $\spec R$, the set of $R$-complexes $M$ with $\supp RM\subseteq U$ is a localizing subcategory of $\dcat R$. In particular if $M\builds[R] N$, then $\supp RM\supseteq \supp RN$. It follows that if $M$ is proxy small, then $\supp RM$ is a closed subset of $\spec R$. 
\end{chunk}

The observation above relating supports to building has a  converse, established by Neeman~\cite[Theorem~2.8]{Neeman:1992b}.

\begin{chunk}
\label{ch:Neeman}
If $M,N$ are objects in $\dcat R$ with $\supp RM\supseteq \supp RN$, then $M\builds[R]N$. 
\end{chunk}

Using this result and \eqref{eq:compact} Neeman deduces the result below concerning finite building, first proved by Hopkins~\cite{Hopkins:1987} using different techniques.

\begin{chunk}
\label{ch:HN}
If $M,N$ are small objects in $\dcat R$ with $\supp RM\supseteq \supp RN$, then $M\fbuilds[R]N$. 
\end{chunk}

The preceding results imply the following characterization of proxy small objects.

\begin{corollary}
\label{co:proxy-support}
Let $R$ be a commutative noetherian ring and $M$ an $R$-complex. Then the following are equivalent
\begin{enumerate}[\rm(1)]
\item $M$ is proxy small;
\item $M$ finitely builds a small object with support equal to $\supp RM$;
\item $\supp RM$ is a closed subset of $\spec R$ and $M$ finitely builds the Koszul complex on a finite subset $\bsx$ of $R$ for which $V(\bsx) = \supp RM$. \qed
\end{enumerate}
\end{corollary}

\section{Hochschild cohomology}
In this section we discuss the (derived) enveloping algebras and Hochschild cohomology of dg $R$-algebras; \cite{Avramov/Iyengar/Lipman/Nayak:2010} is a suitable reference for this material. We are going to be interested in two aspects: One is Hochschild cohomology as a source of operators on the derived category of $A$. The other is the smallness and proxy smallness of $A$ as a module over $\env RA$. In what follows dg algebras will be assumed to be graded-commutative: $a\cdot b = (-1)^{|a||b|}b\cdot a$ for $a,b$ in $A$. 

\begin{chunk}
\label{ch:over}
Given morphisms of dg $R$-algebras $\beta\colon B\to A$ and $\zeta\colon C\to A$, a morphism $\phi\colon B\to C$ is \emph{over} $A$ if $\zeta \phi = \beta$. We say that $B$ and $C$ are quasi-isomorphic \emph{over} $A$ to mean that there is a zig-zag of quasi-isomorphisms over $A$ linking $B$ to $C$. Given the discussion in \ref{ch:ascent-defn} in this situation it is easy to see that $A$ is small, respectively proxy small, in $\dcat B$ if and only if it is small, respectively proxy small, in $\dcat C$.
\end{chunk}

\begin{chunk}
\label{ch:enveloping}
Let $A$ be a dg $R$-algebra and 
\[
\env RA\colonequals A \lotimes_R A\,,
\]
its derived enveloping algebra. When the graded $R$-module underlying $A$ is flat  the canonical map $\env RA \xra{\sim} A\otimes_RA$ is a quasi-isomorphism.

A dg $R$-algebra $A'$ is \emph{homotopically flat} if  tensoring with $A'$ over $R$ preserves quasi-isomorphisms. If $\ve'\colon A'\thra A$ is a homotopically flat dg $R$-algebra resolution of $A$ then  $\env RA$ is represented by $A'\otimes_R A'$. Different homotopically flat resolutions yield dg algebras that are quasi-isomorphic over $A$:  If $\ve''\colon A''\thra A$ is another homotopically flat resolution, then $A'\otimes_R A'$ and $ A''\otimes_RA''$ are quasi-isomorphic over $A$. For this reason we write $\mu\colon \env RA\to A$ to denote a representative of the map 
\[
A'\otimes_R A' \xra{ \ve'\otimes \ve'} A\otimes_R A \xra{\ a\otimes b \mapsto ab\ } A\,.
\]

From \ref{ch:over} it follows that the property that $A$ is small, or proxy small, in $\dcat{\env RA}$ is independent of the choice of a homotopically flat resolution of $A$. In fact, this condition is equivalent to $A$ being small, respectively, proxy small, as a dg $B$-module for any dg $R$-algebra $B$ quasi-isomorphic to $\env RA$ over $A$.
\end{chunk}

\begin{chunk}
\label{ch:hh-action}

The Hochschild, or Shukla, cohomology of a dg $R$-algebra $A$ with coefficients in a dg $A$-module $M$ is 
\[
\HCM RAM \colonequals \Ext{\env RA}{A}{M}\,,
\]
where $A$ is viewed as a dg $\env RA$-module via $\mu$. We abbreviate $\HCM RAA$ to $\HC RA$. This is a graded-commutative $R$-algebra. In what follows we exploit the fact that it acts on $\dcat A$, in the sense of \cite{Krause/Ye:2011}. This action comes about as follows: For any class $\alpha$ in $\HC RA$ and $M$ a dg $A$-module, let 
\[
\chi_M(\alpha) \colon M\to \shift^{|\alpha|}M
\]
be the morphism in $\dcat{A}$ defined by the commutative diagram
\[
\begin{tikzcd}[row sep=4mm]
	A\lotimes_AM  \ar[r,"\alpha \lotimes M "]\ar[d,"\simeq"'] & \shift^{|\alpha|}A\lotimes_AM \ar[d,"\simeq"]\\
	M\ar[r,"\chi_M(\alpha)"] & \shift^{|\alpha|}M\,.
\end{tikzcd}
\]
Thus we get a homomorphism of graded $R$-algebras
\[
\chi_M\colon \HC RA\lra \Ext AMM\,,
\]
called the \emph{characteristic map of $M$}. We denote this $\chi^A_M$ when the dg algebra $A$ needs to be specified. This map has the property that for $N$ in $\dcat A$ and element $\zeta \in \Ext AMN$ one has 
\[
\chi_N(\alpha) \zeta = (-1)^{|\alpha||\zeta|} \zeta \chi_M(\alpha) \,.
\]
In particular $\chi_M(\alpha)$ lies in the graded-center of $\Ext AMM$; see \cite{Krause/Ye:2011} for details. 
\end{chunk}

\begin{chunk}
Fix an $\alpha$ in $\HC RA$ and an $M$ in $\dcat A$. We write $\kos M{\alpha}$ for $\cone(\chi_M(\alpha))$, so there is an exact triangle 
\[
M\xra{\ \chi_M(\alpha)\ }\shift^{|\alpha|}M\lra \kos M{\alpha}\lra
\]
in $\dcat A$. The result below is one of the main reasons for our interest for the action of Hochschild cohomology on $\dcat A$. We do not know if such a statement holds when $\alpha$ is an arbitrary element in the center of $\dcat A$. 
\end{chunk}

\begin{lemma}
\label{le:proxy-central}
If $M$ (finitely) builds $N$, then $\kos M{\alpha}$ (finitely) builds $\kos{N}{\alpha}$. In particular, if $M$ is proxy small then so is $\kos M{\alpha}$.
\end{lemma}

\begin{proof}
The key point is that the action of $\alpha$ on $\dcat A$ is induced by a tensor product:
\[
\kos M{\alpha} = M\lotimes_A {\kos A{\alpha}}\,.
\]
Hence it commutes with exact triangles, retracts, and (possibly infinite) direct sums. It also preserves small objects. Then \ref{ch:proxy-tensor-closed} implies the desired result.
\end{proof}

\begin{chunk}
\label{ch:torsion}
Given an ideal $\fa$ of $\HC RA$, an $\HC RA$-module is \emph{$\fa$-power torsion} if each of its elements is annihilated by a power of $\fa$. When the ideal $\fa$ can be generated by finitely many elements, say $a_1,\dots,a_n$, and a module is $\fa$-power torsion if and only if it is $(a_i)$-power torsion for each $i$.

\begin{lemma}
\label{le:torsion}
Let $N$ be a dg $A$-module and $\fa\subseteq \HC RA$ an ideal. If a dg $A$-module $W$ is such that the $\HC RA$-module $\Ext AWN$ is $\fa$-power torsion, then so is $\Ext AMN$ for any $M$ finitely built from $W$.
\end{lemma}

\begin{proof}
The subcategory of $\dcat A$ with objects $L$ for which the $\HC RS$-module $\Ext ALN$ is $\fa$-power torsion is thick. This implies the desired result.
\end{proof}
\end{chunk}

Next we record a computation of Hochschild cohomology that will be often used; for example, see Lemma~\ref{le:lift} and especially its proof. 

\begin{lemma}
\label{le:atiyah}
Let $\vf\colon R\to S$ be a surjective homomorphism of commutative rings with kernel $I$. For each $S$-module $M$ there is an isomorphism of $S$-modules
\[
\delta^\vf(M) \colon \Hom_{S}(I/I^2,M) \xra{\ \cong\ } \HCM[2]RSM\,,
\]
functorial in $M$. Moreover, given surjective homomorphisms of rings $R\xra{\wt\vf}{\wt S}\xra{\dot\vf} S$ with $\dot\vf\wt\vf =\vf$, for $\wt I=\Ker(\wt \vf)$ the following diagram is commutative
\begin{equation}
\label{eq:atiyah2}
\begin{tikzcd}
\Hom_{\wt S}({\wt I}/{\wt I}^2,M) \ar[d,swap,"\delta^{\wt\vf}(M)"] \ar[r,leftarrow] & \Hom_{S}(I/I^2,M) \ar[d,"\delta^{\vf}(M)"] \\
\HCM[2]R{\wt S}M \ar[r,leftarrow] & \HCM[2]RSM
\end{tikzcd}
\end{equation}
Here the $S$-module $M$ is viewed as an $\wt S$-module by restriction of scalars along $\dot\vf$.
\end{lemma}

\begin{proof}
The map $\delta^\vf(M)$ is part of a family induced by the universal Atiyah class of $\vf$ and involves the cotangent complex; see \cite[Section~5]{Briggs/Iyengar:2020}. We only need $\delta^\vf(M)$ which can be defined quite simply: The multiplication map $\mu\colon \env RS\to S$ embeds in an exact triangle 
\[
J \lra \env RS \xra{\ \mu\ } S \lra 
\]
in $\dcat{\env RS}$. For any $S$-module $M$ viewed as an $\env RS$-module via $\mu$, the exact triangle above induces isomorphisms
\[
\Ext[i]{\env RS}JM \cong \Ext[i+1]{\env RS}SM \quad \text{for $i\ge 1$.}
\]
We claim that $\Ext[1]{\env RS}JM\cong \Hom_S(I/I^2,M)$; with this identification $\delta^{\vf}(M)$ is the isomorphism above for $i=1$. The stated functoriality is easily verified.

As to the claim, as $\vf$ is surjective the natural map $\HH 0{\env RS}=S\otimes_RS \to S$ is an isomorphism, so from the exact triangle above we obtain
\[
\HH iJ =
\begin{cases}
0 & \text{for $i\le 0$}\\
\Tor[i]RSS & \text{for $i\ge 1$}\,.
\end{cases}
\]
It is a standard computation that $\Tor[1]RSS\cong I/I^2$ therefore truncation yields the first isomorphism below
\[
\Ext[1]{\env RS}JM \cong \Ext[0]{\env RS}{I/I^2}M \cong \Hom_S(I/I^2,M)\,.
 \]
The second one holds as the action of $\env RS$ on $I/I^2$ and $M$ factors through $S$.
\end{proof}

\section{Surjective maps}

In this section we prove Theorem~\ref{ithm:local} from the introduction. Throughout $R$ will be a commutative noetherian ring. 

\begin{chunk}
\label{ch:lci-surjective}
A surjective homomorphism $\vf\colon R\to S$ is \emph{complete intersection} if $\Ker(\vf)$ can be generated by a regular sequence; it is \emph{locally} complete intersection if  for each prime $\fq\in \spec S$, the map $\vf_\fq$ is complete intersection. There is no distinction between the conditions when $R$ is local. It was proved in \cite[Theorem~9.1]{Dwyer/Greenlees/Iyengar:2006b} that proxy smallness ascends and descends, in the sense of \ref{ch:ascent-defn}, along complete intersection maps. We prove the converse as part of the result below:
\end{chunk}

\begin{theorem}
\label{th:ci-local}
Let $\vf\colon R \to S$ be a surjective homomorphism of commutative noetherian rings.
The following conditions on $\vf$ are equivalent:
\begin{enumerate}[\rm(1)]
\item \label{th:ci-local1} $\vf$ is locally complete intersection;
\item \label{th:ci-local2} $S\fbuilds[\env RS] \env RS$;
\item \label{th:ci-local3} $S$ is proxy small in $\dcat{\env RS}$ and $\Tor[i] RSS = 0$ for all $i \gg 0$;
\item \label{th:ci-local4} $\pdim_RS$ is finite and proxy smallness ascends along $\vf$;
\item \label{th:ci-local5} $\pdim_RS$ is finite and proxy smallness of objects in $\dcatfl S$ ascends along~$\vf$.
\end{enumerate}
\end{theorem}

The condition that $\pdim_RS$ is finite is equivalent to $S$ being small in $\dcat R$, so condition \ref{th:ci-local4} involves only the structure of the appropriate derived categories as abstract triangulated categories. 

The proof of Theorem~\ref{th:ci-local} takes some preparation and is given in \ref{ch:proof-local-ci}. It builds on ideas from \cite{Pollitz:2019} and extends that results therein, as is explained in \ref{ch:ci-rings}.

\begin{chunk}
\label{ch:delta}
Let $\vf\colon (R,\fm,k)\to S$ be a surjective map of local rings, and let $\epsilon\colon S\to k$ be the canonical surjection. It induces a morphism of dg $S$-algebras $\env RS \to S\lotimes_Rk$. Recall the  standard diagonal isomorphism
\[
(M \otimes_R N) \otimes_{S \otimes_R S} S \xra{\cong} M \otimes_S N\,.
\]
A derived version of this isomorphism yields quasi-isomorphisms of dg algebras
\[
(S\lotimes_Rk)\lotimes_{\env RS}S \simeq S\lotimes_Sk \simeq k\,.
\]
This map and adjunction yield the isomorphism in the definition of the following homomorphism of $S$-modules
\[
\psi^S\colon \HCM RSk \cong \Ext[]{S\lotimes_Rk}kk \lra \Ext[]Skk\,.
\]
The map heading right is induced by restriction along the morphism of dg algebras $S\to S\lotimes_Rk$, and its compatibility with the augmentations to $k$.

It is not hard to verify that the composition of the maps
\[
\HC RS\xra{\ \HCM RS{\epsilon}\ } \HCM RSk \xra{\ \psi^S } \Ext[]Skk
\]
is nothing but the characteristic map $\chi_k$ described in~\ref{ch:hh-action}. 
\end{chunk}

The next results concern the following scenario.

\begin{chunk}
\label{ch:factorization}
Let $\vf\colon R\to S$ be a surjective local homomorphism admitting a factorization 
\[
(R,\fm,k) \xra{\ \wt\vf\ } \wt S\xra{\ \dot{\vf}\  }S
\]
such that for $\wt I\colonequals\Ker \wt \vf $ and $I\colonequals \Ker \vf$  the induced map $\wt I/\fm \wt I \to I/\fm I$ is injective.

For any element $s$ in $\HC[2]R{\wt S}$, we write $\kos ks$ for the mapping cone of the element $\chi^{\wt S}_k(s)$ in $\Ext[2]{\wt S}kk$. Restriction induces a functor
\[
{\dot\vf}_*\colon \dcat S\to \dcat{\wt S}\,,
\]
of triangulated categories. 

\end{chunk}
 
\begin{lemma}
\label{le:lift}
With notation and hypotheses as in \ref{ch:factorization} the induced maps
\[
\Ext[2]{S}kk \xra{\ \Ext[2]{\dot\vf}kk\ } \Ext[2]{\wt S}kk \xleftarrow{\ \chi^{\wt S}_k\ } \HC[2]R{\wt S} 
\]
one has an inclusion $ \Image(\Ext[2]{\dot\vf}kk)\supseteq \Image(\chi^{\wt S}_k)$.
\end{lemma}

\begin{proof}
The essence of the proof  is a commutative diagram of $k$-vector spaces
\[
\begin{tikzcd}
\Hom_S(I/I^2,k) \ar[d, "\cong", "\delta^{\vf}(k)" swap]\ar[twoheadrightarrow,r] 
	&\Hom_{\wt S}({\wt I}/{\wt I}^2,k) \ar[d, "\cong","\delta^{\wt\vf}(k)" swap] \ar[leftarrow,r]
		&\Hom_{\wt S}({\wt I}/{\wt I}^2,\wt S) \ar[d, "\cong", "\delta^{\wt\vf}(\wt S)" swap]\\
\HCM[2]RSk \ar[d,swap,"\psi^{S}"] \ar[twoheadrightarrow,r]& \HCM[2]R{\wt S}k \ar[d,swap,"\psi^{\wt S}"] & \HC[2]R{\wt S}\ar{l} \ar[dl,"\chi^{\wt S}_k"] \\
\Ext[2]Skk\ar[r,swap,"{\Ext[2]{\dot\vf}kk}"] & \Ext[2]{\wt S}kk & 
\end{tikzcd}
\]
with surjective maps and isomorphisms as indicated. The $\delta$ maps are from Lemma \ref{le:atiyah}. Given this the desired inclusion can be verified by chasing around the diagram. 

As to the commutativity of the diagram: The squares in the top row are commutative by the functoriality of $\delta^{\wt\vf}(-)$ with respect to the ring argument and the module argument. The vertical maps $\psi^{S}$ and $\psi^{\wt S}$ are from \ref{ch:delta}, and the commutativity of that square is by functoriality of the construction, which is readily verified. The commutativity of the triangle has been commented on already in \ref{ch:delta}.

It remains to verify the surjectivity of the map in the top left square: Since the map of $k$-vector spaces ${\wt I}/{\fm \wt I} \to I/{\fm I}$ is injective, it is split-injective. Therefore applying $\Hom_k(-,k)$ yields surjectivity of the map below
\[
\Hom_k(I/\fm I,k) \twoheadrightarrow \Hom_k(\wt I/\fm\wt I,k)\,.
\]
This identifies with the surjection in the top left square via the adjunction isomorphism $\Hom_S(N,k) \cong \Hom_k(N/\fm N,k)$. This completes the proof of the claims about the commutative diagram above, and hence that of the result.
\end{proof}

The result below is a crucial input in the proof that \ref{th:ci-local5}$\Rightarrow$\ref{th:ci-local1} in Theorem~\ref{th:ci-local}.

\begin{lemma}
\label{le:key}
With hypotheses as in \ref{ch:factorization}, given an element $s$ in $\HC[2]R{\wt S}$ there exists an element $t$ in $\Ext[2]Skk$ such that ${\dot\vf}_*(\kos k{t}) \cong \kos ks$ in $\dcat {\wt S}$. Moreover, for any such $t$, the element $s^2$ annihilates $\Ext{\wt S}{{\dot\vf}_*(\kos k{t})}-$.
\end{lemma}

\begin{proof}
By Lemma~\ref{le:lift}, there exists an element $t$ in $\Ext[2]Skk$ whose image under $\Ext[2]{\dot\vf}kk$ equals $\chi^{\wt S}_k(s)$. This means that in $\dcat{\wt S}$ there is a commutative diagram
\[
\begin{tikzcd}
k \ar[d,swap,"\cong"] \ar[r,"\chi^{\wt S}_k(s)"] & \shift^2k \ar[d,"\cong"] \\
k\ar[r,"{{\dot\vf}_*(t)}"] & \shift^2k
\end{tikzcd}
\]
As ${\dot\vf}_*$ is exact, the first part of the statement follows. The second part is clear.
\end{proof}

In the proof of Theorem \ref{th:ci-local} we also need a criterion for detecting small complexes through the action of Hochschild cohomology. 

\begin{chunk}
\label{ch:variety}
Let $\vf\colon (R,\fm,k)\to S$ be a  surjective local complete intersection map with kernel  $I$. Set $\mcN\colonequals \Hom_S(I/I^2,S)$, this is the normal module of $\vf$. The Hochschild cohomology algebra $\HC RS$ is graded-commutative so the map $\delta^\vf(S) \colon \mcN\to \HC[2] RS$ described in Lemma~\ref{le:atiyah} induces a homomorphism of $S$-algebras 
\[
\Sym_S(\mcN) \lra \HC RS\,.
\]
Since $\vf$ is complete intersection the $S$-module $\mcN$ is free of rank the codimension of $S$ in $R$ and the map above is bijective; see \cite[Proposition~2.6]{Avramov/Buchweitz:2000a}. In particular, the ring $\HC RS$ is noetherian. From \emph{op.~cit.} one also gets that for any $S$-module $M$ the natural map is an isomorphism:
\begin{equation}
\label{eq:hh-ci}
\HC RS\otimes_S M \xra{\ \cong\ } \HCM RSM\,.
\end{equation}
Later on we will use the fact that this isomorphism is functorial in $M$. 
\end{chunk}

\begin{lemma}
\label{le:Yo}
Let $\vf$ be as in \ref{ch:variety} and $M$ an $S$-complex. If $M$ is small in $\dcat R$ and for some generating set $s_1,\dots,s_c$ for the $S$-module $\HC[2]RS$ the $\HC RS$-module $\Ext SMk$ is $(s_i)$-power torsion for each $i$, then $M$ is small in $S$.
\end{lemma}

\begin{proof}
Since $M$ is small in $\dcat R$, the $\HC RS$-module $\Ext SMk$ is finitely generated; see, for example, \cite[Corollary~6.2]{Avramov/Buchweitz:2000a}. Since $\HC[2]RS$  generates $\HC RS$ as an $S$-algebra, the hypothesis implies that  $\Ext SMk$ is  $\HC[\ges 1]RS$-power torsion, and hence $\Ext[i]SMk=0$ for $i\gg 0$. Thus $M$ is small in $S$. 
\end{proof}

\begin{chunk}
\label{ch:proof-local-ci}
\begin{proof}[Proof of Theorem~\ref{th:ci-local}]
\ref{th:ci-local1}$\Rightarrow$\ref{th:ci-local3} 
As $\vf$ is locally complete intersection its flat dimension is finite so the $R$-module $\hh{\env RS}=\Tor RSS$ is finitely generated; in particular $\Tor [i]RSS=0$ for $i\gg 0$. It remains to check that $S$ is proxy small in $\dcat{\env RS}$. When $R$ is a local ring $\Ker(\vf)$ is generated by a regular sequence and then the desired result is contained in the \emph{proof} of \cite[Theorem~9.1]{Dwyer/Greenlees/Iyengar:2006b}. The hypothesis that $\vf$ is complete intersection is local on $\spec S$, meaning that $\vf$ is complete intersection if and only if the map of local rings $\vf_\fq$ is complete intersection for each $\fq$ in $\spec S$. This is by defintion.
We claim that the conclusions that $S$ is proxy small in $\dcat{\env RS}$ is also local on $\spec S$.

Indeed the $R$-module $\Tor RSS$ is finitely generated so Proposition~\ref{pr:local-to-global} applies to the dg $S$-algebra $A\colonequals \env RS$ and $M\colonequals S$ to yield that $S$ is proxy small in $\dcat{\env RS}$ if and only if $S_\fq$ is proxy small in $\dcat{(\env RS)_\fq}$. It remains to observe that $(\env RS)_\fq\cong \env{R_{\fq\cap R}}{(S_\fq)}$.

\ref{th:ci-local2}$\Leftrightarrow$\ref{th:ci-local3} Since $R\to S$ is surjective, $\HH{0}{\env RS}=S$. Thus Lemma \ref{le:h0} part \ref{le:h0part2} yields the desired equivalences.

\ref{th:ci-local2}$\Rightarrow$\ref{th:ci-local4} The assumption that $S\fbuilds[\env RS] \env RS$ implies that proxy smallness ascends along $\vf$; see \cite[Theorem~8.3]{Dwyer/Greenlees/Iyengar:2006b}. It remains to verify that $S$ is small in $\dcat R$. Since $S\fbuilds[\env RS] \env RS$, for any $S$-module $M$ applying $(-)\lotimes_SM$ yields $M\fbuilds[R] (S\lotimes_RM)$ hence $\hh{S\lotimes_RM}$ is bounded whenever $\hh M$ is bounded. Since $\hh{S\lotimes_RM}=\Tor RSM$ it follows that $\fdim_R S<\infty$, and as $\vf$ is surjective we can conclude that $S$ is small over $R$.

\ref{th:ci-local4}$\Rightarrow$\ref{th:ci-local5} This  is a tautology.

\ref{th:ci-local5}$\Rightarrow$\ref{th:ci-local1} The desired conclusion can be checked locally at the maximal ideals of $S$, and the hypothesis is easily seen to descend to localization at any such ideal. Thus we may assume $\vf\colon (R,\fm,k) \to (S,\fn,k)$ is a surjective local homomorphism. 

Choose a maximal regular sequence $\bsx$ in $\Ker(\vf)\setminus \fm\Ker(\vf)$ and set $\wt{S}\colonequals R/(\bsx)$. The map $\vf$ factors as
\[
\begin{tikzcd}
R \ar[r,"{\wt\vf}"] \ar[rr, bend right=20, "{\vf}" swap] & \wt S \ar[r,"{\dot\vf}"] & S\,,
\end{tikzcd}
\]
where $\wt\vf$ is complete intersection and $\Ker(\dot \vf)$ contains only zero-divisors; the latter condition implies that either $\wt S=S$ or $S$ is \emph{not small} in $\dcat {\wt S}$; see \cite[Corollary~1.4.7]{Bruns/Herzog:1998}. We shall prove that under the hypothesis $S$ is small in $\dcat{\wt S}$ yielding that $\vf=\wt\vf$ and hence that $\vf$ is a complete intersection, as desired.

The argument involves a series of reductions. We are now in context of \ref{ch:factorization}, and we keep the notation from there. The desired conclusion is that ${\dot\vf}_*(S)$ is small. Let $K$ be the Koszul complex on a set of generators for the maximal ideal of $S$. By a standard reduction recalled in \cite[Remark~5.6]{Dwyer/Greenlees/Iyengar:2006b} it suffices to verify that the $\wt S$-complex ${\dot\vf}_*(K)$ is small. We do so by checking that the hypotheses of Lemma~\ref{le:Yo} hold for the complete intersection $\wt\vf$ and $M\colonequals {\dot\vf}_*(K)$.

As $\vf_*(S)$ is small in $\dcat R$, by assumption, so is $\vf_*(K)$ for it is finitely built out of $\vf_*(S)$. Next, fix an element $s$ in $\HC[2]R{\wt S}$ and let $t$ be the  element in $\Ext[2]Skk$ provided by Lemma~\ref{le:key}. Thus there is an isomorphism
\[
{\dot{\vf}}_*(\kos k{t}) \cong \kos k{s} \quad \text{in $\dcat{\wt S}$.}
\]
Since $k$ is proxy small in $\dcat{\wt S}$ so is $\kos k{s}$; this uses the fact $s$ comes from $\HC R{\wt S}$; see~Lemma~\ref{le:proxy-central}. Thus ${\dot{\vf}}_*(\kos k{t})$ is proxy small in $\dcat{\wt S}$. Since $R\to \wt S$ is complete intersection, this implies that the $R$-complex
\[
\vf_*(\kos k{t}) \cong \wt{\vf}_*({\dot{\vf}}_*(\kos k{t}))
\]
is proxy small in $\dcat{R}$ by  \cite[Theorem~9.1]{Dwyer/Greenlees/Iyengar:2006b}. 

Evidently the $S$-complex $\kos k{t}$ is in $\dcatfl S$ hence the hypothesis of Theorem~\ref{th:ci-local}\ref{th:ci-local5}  yields that $\kos kt$ is proxy small in $\dcat S$. So $\kos kt$ finitely builds $K$ in $\dcat S$; this is by Corollary~\ref{co:proxy-support}. It follows that ${\dot\vf}_*(\kos k{t})$ finitely builds ${\dot\vf}_*(K)$. We can now apply the second part of Lemma~\ref{le:key} and Lemma \ref{le:torsion} to deduce that the $\HC R{\wt S}$-module $\Ext{\wt S}{{\dot\vf}_*(K)}k$ is $(s)$-power torsion. Then Lemma~\ref{le:Yo}  yields that ${\dot\vf}_*(K)$ is small. 

This wraps up the proof that \ref{th:ci-local5}$\Rightarrow$\ref{th:ci-local1} and thereby that of Theorem~\ref{th:ci-local}.
\end{proof}
\end{chunk}

\begin{chunk}
In Theorem~\ref{th:ci-local}, it seems plausible one can relax the hypothesis in \ref{th:ci-local4} to: $\pdim_RS$ is finite and any $M$ in $\dbcat S$ that is \emph{small} in $\dcat R$ is proxy small in $\dcat S$. This condition already implies $\vf$ is quasi-Gorenstein, that is to say: $\RHom RSR \simeq \shift^c S$ in $\dcat S$, where $c=\dim R-\dim S$; see \cite[Theorem~6.7]{Dwyer/Greenlees/Iyengar:2006b}. 
\end{chunk}

A shortcoming in the statement of Theorem~\ref{th:ci-local}\ref{th:ci-local5} is that it involves all of $\dcatfl S$. When $R$ is local a careful reading of the proof shows that one needs to check the hypothesis on finitely many complexes. In fact, these complexes can specified in advance. This is clarified in the discussion below.

\begin{chunk}
\label{ch:pi-test}
Let $(R,\fm,k)$ be a local ring and let $\pi(R)$ denote its homotopy Lie algebra; see \cite[Chapter~10]{Avramov:2010}. This is a naturally constructed graded subspace of the graded $k$-vector space $\Ext Rkk$, so an element $\zeta \in \pi^n(R)$ represents a morphism $k\to \shift^nk$ in $\dcat R$. For what follows we care only about $\pi^2(R)$, and that can be made explicit.

Let $\rho \colon Q\to R$ be a minimal Cohen presentation of $R$; that is to say, $Q$ is a regular local ring with $\dim Q$ equal to the embedding dimension of $R$. Such a presentation exists when, for example, $R$ is $\fm$-adically complete; this is part of Cohen's structure theorem~\cite[Theorem~A.21]{Bruns/Herzog:1998}. With $J\colonequals \Ker(\rho)$ the image of the composition
\[
\Hom_R(J/J^2,k) \xra[\cong]{\delta^{\rho}(k)} \HCM [2]QRk \xra{\ \psi^R\ } \Ext[2]Rkk
\]
is precisely $\pi^2(R)$, and the composite map is a bijection onto its image; see \cite[Example~10.2.2]{Avramov:2010} or \cite{Sjodin:1976}. 

For any surjective local homomorphism $\vf\colon R\to S$ one gets a map
\[
\pi^2(\vf)\colon \pi^2(S) \lra \pi^2(R)\,,
\]
of $k$-vector spaces, making $\pi^2(-)$ into a functor on the category of local rings and surjective local homomorphisms. Moreover the image of the map
\[
\psi^S\colon \HCM [2]RSk\to \Ext[2]Skk
\]
is $\Ker(\pi^2(\vf))$. This observation is a key ingredient in the proof of the next result.

\begin{theorem}
\label{th:lci-test-complexes}
Let $\vf\colon (R,\fm,k) \to S$ be a surjective homomorphism of local rings with $\pdim_RS$ finite. The following conditions are equivalent:
\begin{enumerate}[\rm(1)]
\item
\label{th:lci-test-complexes1}
$\vf$ is complete intersection;
\item
\label{th:lci-test-complexes2}
For each $t\in \Ker(\pi^2(\vf))$ the $S$-complex $\kos kt$ is proxy small.
\item
\label{th:lci-test-complexes3}
For some generating set $t_1,\dots,t_d$ for the $k$-vector space $\Ker(\pi^2(\vf))$, the $S$-complexes $\kos k{t_i}$ are proxy small.
\end{enumerate}
\end{theorem}

\begin{proof}
\ref{th:lci-test-complexes1}$\Rightarrow$\ref{th:lci-test-complexes2} This follows from Lemma~\ref{le:proxy-central} and the fact that the natural maps
\[
\HC[2]RS\lra \HCM[2]RSk \xra{\ \psi^S \ } \Ker(\pi^2(\vf))
\]
are surjective. The surjectivity of the map on the left is a consequence \eqref{eq:hh-ci} which holds because $\vf$ is complete intersection. The surjectivity of the map on the right has been commented on earlier. 

\ref{th:lci-test-complexes2}$\Rightarrow$\ref{th:lci-test-complexes3} This is a tautology.

\ref{th:lci-test-complexes3}$\Rightarrow$\ref{th:lci-test-complexes1} The proof for this implication follows that of \ref{th:ci-local5}$\Rightarrow$\ref{th:ci-local1} in Theorem~\ref{th:ci-local}. We keep the notation from there. Recall the factorization of $\vf$:
\[
R\xra{\wt\vf}\wt S\xra{\dot\vf} S
\]
with $\wt\vf$ defined by a maximal regular sequence in $\Ker(\vf)\setminus\fm\Ker(\vf)$. Let $K$  be the Koszul complex on a finite set generating set for the maximal ideal of $\wt S$. As before the strategy is to prove that ${\dot\vf}_*(K)$ is small in $\dcat{\wt S}$. The main point is to verify that the $\HC R{\wt S}$-module $\Ext{\wt S}{{\dot\vf}(K)}k$ is $(s)$-power torsion for all $s$ from a generating set for the $\wt S$-module $\HC[2]R{\wt S}$, for then we can invoke Lemma~\ref{le:Yo}.

The new observation here is that since ${\dot\vf}_*(K)$ is induced from $\dcat S$  we need only worry about a subset of a generating for the $\wt S$-module $\HC[2]R{\wt S}$. Indeed,  consider  maps
\[
\Ext[2]Skk \xra{\ \Ext[2]{\dot\vf}kk\ }\Ext[2]{\wt S}kk \xla{\ \chi^{\wt S}_k\ } \HC[2]R{\wt S}\,.
\]
Since the action of $\HC [2]R{\wt S}$ on $\Ext S{{\dot\vf}_*(K)}k$ factors through $\Ext[2]{\wt S}kk$ it suffices to verify the following statement.

\begin{claim*}
There exist elements $s_1,\dots,s_d$ in $\HC [2]R{\wt S}$ such that 
\begin{enumerate}[\quad\rm(a)]
\item
$\chi^{\wt S}_k(s_1),\dots,\chi^{\wt S}_k(s_d)$ span the image of $\chi^{\wt S}_k$ as a $k$-vector space;
\item
$\chi^{\wt S}_k(s_i)$ is  the image of $t_i$  under the map $\Ext[2]{\dot\vf}kk$, for all $i$.
\end{enumerate}

Once we verify this claim, arguing as in the proof of Theorem~\ref{th:ci-local} leads to the desired conclusion.
\end{claim*}

As to verifying the claim, the functoriality of $\pi^2(-)$ gives a commutative diagram
\[
\begin{tikzcd}
\pi^2(S)\ar[r,"{\pi^2(\dot\vf)}"]	\ar[rr, bend right=20, "{\pi^2(\vf)}" swap]
		&\pi^2(\wt S)\ar[r,"{\pi^2(\wt\vf)}"] & \pi^2(R).
\end{tikzcd}
\]
Thus one gets an induced map $\pi^2(\dot\vf)$ in the following diagram of $k$-vector spaces:
\[
\begin{tikzcd}
\HCM[2]RSk \ar[d, twoheadrightarrow, "{\psi^{S}}" swap] \ar[r, twoheadrightarrow] 
	& \HCM[2]R{\wt S}k \ar[d, twoheadrightarrow, "{\psi^{\wt S}}" swap] 
	      &\HC[2]R{\wt S} 	\ar[l, twoheadrightarrow]	\ar[dl, "\chi^{\wt S}_k"]\\
\Ker(\pi^2(\vf)) \ar[r, "{\pi^2(\dot\vf)}"] 
	&\Ker(\pi^2(\wt\vf)) \,.
\end{tikzcd}
\]
The commutative square is from Lemma~\ref{le:lift}.  The surjectivity of the map pointing left is from \eqref{eq:hh-ci} applied to $\wt\vf$.  It remains to take $s_1,\dots,s_d$ to be any preimages of $\pi^2(\dot\vf)(t_1),\dots, \pi^2(\dot\vf)(t_d)$ under $\chi^{\wt S}_k$.
\end{proof}
\end{chunk}

\begin{chunk}
\label{ch:ci-rings}
A local ring $(S,\fn)$ is \emph{complete intersection} if for some (equivalently, any) Cohen presentation $\vf\colon R\to \widehat{S}$ of the $\fn$-adic completion of $S$, the map $\vf$ is complete intersection; see \cite[Section~2.3]{Bruns/Herzog:1998}. A commutative noetherian ring $S$ is \emph{locally complete intersection} if it is complete intersection at each prime ideal of $S$. Theorem~\ref{th:ci-local} applied locally to Cohen presentations recovers a characterization of complete intersections established in \cite[Theorem~5.2]{Pollitz:2019} and \cite[Corollary~5.6]{Letz:2020}.

\begin{corollary}
\label{cor:josh}
Let $S$ be a commutative noetherian ring. The following conditions are equivalent:
\begin{enumerate}[\quad\rm(1)]
\item $S$ is locally complete intersection;
\item Each object in $\dbcat S$ is proxy small;
\item Each object in $\dcatfl S$ is proxy small. \qed
\end{enumerate}
\end{corollary}

A natural question is whether the locally complete intersection property  is detected by proxy smallness of modules, that is to say, $S$-complexes with homology concentrated in a single degree; see \cite{Briggs/Grifo/Pollitz:2020} for some positive results.

\end{chunk}

\section{Factorization of locally complete intersection maps}
\label{sec:applications}
In this section we demonstrate the strength of Theorem~\ref{th:ci-local} by deducing some fundamental results on complete intersection rings and maps. We restrict ourselves to treating surjective local maps for reasons laid out in Remark~\ref{re:general-lci}. As a warm up we recover a well-known result  tracking the complete intersection property along maps of finite projective dimension. 

\begin{corollary}
\label{cor:ci-ad}
Let $\vf\colon R\to S$ be a surjective local map with $\pdim_RS$ finite. Then the ring $S$ is complete intersection if and only if the ring  $R$ and the map $\vf$ are complete intersection.
\end{corollary}

\begin{proof}
Assume $S$ is complete intersection. By Corollary~\ref{cor:josh} objects in $\dcatfl S$ are proxy small so the condition in Theorem~ \ref{th:ci-local}\ref{th:ci-local5} holds trivially and hence $\vf$ is complete intersection. To verify that $R$ is  complete intersection it suffices to verify that any nonzero $M$ in $\dcatfl R$ finitely builds a nonzero small object; this is where the hypothesis that $\hh M$ has finite length is used. Since $R\fbuilds S$, by hypothesis, one gets that $M\fbuilds[R] S\lotimes_RM$. View $S\lotimes_RM$ as an object in $\dcat S$. Since $\pdim_RS$ is finite, the length of the $S$-module  $\hh{S\lotimes_RM}$ is finite; it is also nonzero by Nakayama's lemma. Since $S$ is complete intersection $S\lotimes_RM$ builds a nonzero small object in $\dcat S$, and that object is also small in $R$, since $\pdim_RS$ is finite.  This is the desired conclusion.

Suppose $R$ and $\vf$ are complete intersection. The hypothesis on $R$ implies that any $M\in\dbcat S$ is proxy small in $\dcat R$, by Corollary~\ref{cor:josh}, and then the hypothesis on $\vf$ implies that $M$ is also proxy small in $\dcat S$, by Theorem~\ref{th:ci-local}. Another application of Corollary~\ref{cor:josh} yields that $S$ is complete intersection.
\end{proof} 

The forward implication in the result below is easy to prove directly from the definitions; it is equally simple to deduce it from Theorem~\ref{th:ci-local}. The converse statement is due to Avramov~\cite[5.7]{Avramov:1999}. The proof in \emph{op.~cit.} is complicated and involves nontrivial properties of Andr\'e--Quillen homology. The proof presented below is more elementary and natural from the perspective of ascent of proxy smallness.

\begin{corollary}
\label{cor:lcifac}
Let $R\xra{\vf}S\xra{\psi}T$ be surjective local homomorphisms.

If the maps $\vf$ and $\psi$ are complete intersection, then so is $\psi\vf$. The converse holds if, in addition, $\pdim_ST$ is finite.
\end{corollary}

\begin{proof}
We make repeated use of Theorem~\ref{th:ci-local} without specific reference.

Suppose $\vf$ and $\psi$ are complete intersection. Since $R\fbuilds[R]S$ and $S\fbuilds[S]T$ and hence $S\fbuilds[R]T$, it follows that $R\fbuilds[R]T$, that is to say, $\pdim_RT$ is finite. Moreover  given an $M\in\dcat T$, if $M$ is proxy small in $\dcat R$, then it is proxy small in $\dcat S$, since $\vf$ is complete intersection, and hence in $\dcat T$, since $\psi$ is complete intersection. Thus proxy smallness ascends along $\psi\vf$ and hence it is a complete intersection.

Now suppose $\psi\vf$ is complete intersection and $\pdim_ST$ is finite, that is to say, $S\fbuilds[S]T$. The first condition implies $R\fbuilds[R]T$ and then the second one implies $R\fbuilds[R]S$, by \cite[Remark~5.6]{Dwyer/Greenlees/Iyengar:2006b}. It follows that $\psi$ is complete intersection: Any $T$-complex that is proxy small in $\dcat S$ is proxy small in $\dcat R$, since $R\fbuilds[R]S$, and hence is proxy small in $\dcat T$, since $\psi\vf$ is complete intersection. 

Now $\psi\vf$ and $\psi$ are both complete intersection; we prove that so is $\vf$. Fix $M$ in $\dcat S$ with nonzero finite length homology such that $M$ is proxy small in $\dcat R$. It suffices to prove that $M$ finitely builds a small object in $\dcat S$ with nonzero homology; this is where we need to use the assumption that $\hh M$ has finite length.

Since $S\fbuilds[S]T$ one gets that $M\fbuilds[S] T\lotimes_SM$. Let $\bsx$ be a finite set of elements in $R$ whose images in $S$ form a minimal generating set for the ideal $\Ker(\psi)$, and let $K$ be the Koszul complex on $\bsx$, with coefficients in $R$. Since $\psi$ is complete intersection, the $S$-complex $K\otimes_RS$ is a resolution of $T$ over $S$. This justifies the second of the following isomorphisms in $\dcat R$:
\[
K\otimes_R M \simeq (K\otimes_RS)\otimes_S M \simeq T\lotimes_SM\,.
\]
The first one is by associativity of tensor products. As $M$ proxy small in $R$ it follows that so is $T\lotimes_SM$. However as $T\lotimes_SM$ is in $\dcat T$, the complete intersection property of $\psi\vf$ implies that $T\lotimes_SM$ is proxy small in $\dcat T$ and hence  in $\dcat S$. Thus $T\lotimes_SM$, and hence also $M$, finitely builds a small $S$-complex with nonzero homology.
\end{proof}

\begin{remark}
In Corollary~\ref{cor:lcifac} one cannot weaken the hypothesis in the converse that $\pdim_ST$ is finite to $\psi$ is proxy small. Indeed, let $R$ be a regular local ring with residue field $k$ and consider surjective local homomorphisms
$R \xra{\vf} S\xra{\psi} k$.
Then $\psi\vf$ is complete intersection, and $\psi$ is proxy small, but $\vf$ is complete intersection if and only if $S$ is, and that need not be the case.
\end{remark}

However  Corollary~\ref{cor:ci-ad} yields the following result.

\begin{corollary}
Let $R\xra{\vf}S\xra{\psi}T$ be surjective local homomorphisms such that $S$ and the map $\psi\vf$ are complete intersection. Then the rings $R$ and $T$, and the map $\vf$, are complete intersection.
\end{corollary}

\begin{proof}
As $\psi\vf$ is complete intersection $\pdim_RT$ is finite.  Since $T$ is proxy small in $\dcat S$, by Corollary~\ref{cor:josh}, from \cite[Theorem~5.5]{Dwyer/Greenlees/Iyengar:2006b} one gets that $\pdim_RS$ is finite. Then as $S$ is complete intersection it follows that $R$ and $\vf$ are complete intersection as well, by Corollary~\ref{cor:ci-ad}. Another application of Corollary~\ref{cor:ci-ad}, now to the map $\psi\vf$, yields that $T$ is complete intersection.
\end{proof}

\section{Essentially of finite type maps}
\label{sec:eft}

\begin{chunk}
\label{ch:smooth}
A morphism $\vf\colon R\to S$ of commutative rings is \emph{essentially of finite type} if it is obtained as the localization of a finitely generated $R$-algebra; that it to say, $\vf$ admits a factorization 
\begin{equation}
\label{eq:smooth-factorization}
R \xra{\ \wt\vf\ } \wt R \xra{\ \dot\vf\ } S
\end{equation}
where $\wt R=U^{-1}R[\bsx]$, where $\bsx\colonequals x_1,\dots,x_n$ are indeterminates, $U$ is a multiplicatively closed subset in $R[\bsx]$, and $\dot\vf$ is surjective. Such a map is \emph{smooth} if $\vf$ is flat, and for each map of rings $R\to l$ with $l$ a field, the ring $l\otimes_RS$ is regular; that is to say, the fibers of $\vf$ are geometrically regular; see \cite{Grothendieck:1967}. For example, the map $\wt\vf$ in the factorization above is smooth. 

The map $\vf$ is \emph{locally complete intersection} if the surjection $\dot\vf \colon \wt R\to S$ is locally complete intersection in the sense of \ref{ch:lci-surjective}. This condition is independent of the factorization of $\vf$; this fact is also implicit in the proof of Theorem~\ref{th:global-eft} below. When $\vf$ is flat, it is locally complete intersection if and only if all its fibers are locally complete intersection rings. For details, see \cite{Grothendieck:1967}.
\end{chunk}

The result below is analogous to a classical characterization of smooth maps, recalled in the introduction, namely: \emph{$\vf$ is smooth if and only if it is flat and $S$ is small in $\dcat{\env RS}$}; see \cite[Proposition~(17.7.4)]{Grothendieck:1967} and \cite[Theorem~1]{Rodicio:1990}. A crucial difference: We do not have to assume $\vf$ is flat, only that $\Tor RSS$ is bounded. This comes with a caveat: In the statement $\env RS$ is the derived enveloping algebra.

\begin{theorem}
\label{th:global-eft}
Let $R$ be a commutative noetherian ring and $\vf\colon R \to S$ a morphism essentially of finite type. The following conditions are equivalent:
\begin{enumerate}[\quad\rm(1)]
\item $\vf$ is locally complete intersection;
\item $S$ is proxy small in $\dcat{\env RS}$ and $\Tor[i] RSS = 0$ for all $i \gg 0$.
\end{enumerate}
\end{theorem}

This result specializes to Theorem~\ref{ithm:eft} from the Introduction.

\begin{proof}
Fix a factorization \eqref{eq:smooth-factorization} of $\vf$. As noted before $\vf$ is locally complete intersection if and only if $\dot\vf$ is locally complete intersection. When this property holds  $\pdim_{\wt R}S$ is finite and hence so is $\pdim_RS$; in particular $\Tor[i]RSS=0$ for $i\gg 0$. Thus in the remainder of the proof we  assume that $\Tor RSS$ is bounded.

Next we reduce to the case where $\vf$ is surjective. To that end consider the morphism of dg algebras
\[
\theta\colon \env RS\lra \env{\wt R}{S} 
\]
induced by $\wt\vf$. The following claim implies that condition (2) holds for $\vf$ if and only if it holds for $\dot\vf$. 

\begin{claim*}
$\Tor {\wt R}SS$ is bounded and proxy smallness ascends and descends along $\theta$.
\medskip
 
The crucial point is that $\theta$ can be obtained by base change of a locally complete intersection map, so it is only a definition away from being itself locally complete intersection; see \cite[Definition~2.4]{Shaul:2020}. Here are the details.

Consider the natural multiplication map $\mu\colon \env R{\wt R}\to \wt R$. The diagonal isomorphism yields a quasi-isomorphism of dg algebras
\[
\wt R \lotimes_{\env R{\wt R}} \env RS = \wt R \lotimes_{\env R{\wt R}} (S\lotimes_R S) \simeq S\lotimes_{\wt R} S = \env {\wt R}S\,.
\]
Thus $\theta$ is the base change of $\mu$ along the morphism $\env R{\wt R}\to \env RS$, that is to say, there is a cofiber square
\[
\begin{tikzcd}
\env R{\wt R} \ar[r,"\mu"]\ar[d, swap, "\dot{\vf}\lotimes_R\dot{\vf}"] & \wt R \ar[d] \\
\env R S \ar[r,swap,"\theta"] & \env {\wt R}S\,.
\end{tikzcd}
\]
Recall that $\wt R$ is a localization of a polynomial ring $R[\bsx]$ where $\bsx\colonequals x_1,\dots,x_n$ are indeterminates. Thus $\env R{\wt R}$ may be taken to be the ordinary tensor product $\wt R\otimes_R\wt R$ and the kernel of the multiplication map $\env R{\wt R}\to \wt R$ is generated by the regular sequence $x_1\otimes 1 - 1\otimes x_1,\dots, x_n\otimes 1 - 1\otimes x_n$. The Koszul complex on this sequence is a dg algebra resolution of $\wt R$ over $\env R{\wt R}$, so we deduce that map $\theta$ can also be realized as an extension by a Koszul complex. The first part of the desired claim is now a standard verification. For the ascent and descent of proxy smallness along $\theta$ see \cite[Remark~9.2]{Dwyer/Greenlees/Iyengar:2006b}.
\end{claim*}

Thus we can replace $\vf$ by $\wt{R} \to S$ and assume that $\vf$ is surjective. At this point we can apply Theorem~\ref{th:ci-local}. 
\end{proof}

\begin{remark}
In the context of Theorem~\ref{th:global-eft}, if $\Tor[i] RSS = 0$ for $i \geq 1$, for example when $\vf$ is flat, we take $\env RS$ to be the usual tensor product $S\otimes_RS$. Then \ref{ch:HN} yields that $S$ is proxy small in $\dcat {\env RS}$ if and only if it finitely builds the Koszul complex on a finite generating set for the kernel of the multiplication map $S\otimes_RS\to S$.
\end{remark}

\begin{remark}
\label{re:general-lci}
Theorem~\ref{th:global-eft} is missing a characterization in terms of the exact functor $\dcat S\to \dcat R$, akin to the one in Theorem~\ref{th:ci-local}\ref{th:ci-local4}. It is not reasonable to expect ascent and descent of the proxy small property along maps that are not finite. There is a notion of proxy smallness with respect to a map, via the surjective part of the smooth-by-surjective factorizations in \eqref{eq:smooth-factorization}, but this is not entirely satisfactory, partly because being essentially of finite type is not a condition that can be characterized purely in terms of categorical properties of derived categories. 

These questions are of interest if one wishes to import the ideas of this paper into stable homotopy theory, where the requiring that a map of commutative ring spectra to be surjective, or essentially of finite type, is not sensible. There are many rich examples of complete intersection like behavior in that context~\cite{Benson/Greenlees/Shamir:2013}, and it would be interesting to develop a notion of ascent of proxy smallness that captures them.
\end{remark}

\section{Morphisms of schemes}
\label{se:schemes}

Let $X$ be a noetherian separated scheme, $\dcat{X}$ its derived category of quasi-coherent sheaves, and $\sdbcat{X}$ the bounded derived category of coherent sheaves. We write $\perf{X}$ for its full subcategory of perfect complexes. As in the affine case, these are precisely the small objects in $\dcat{X}$; see \cite[Proposition~1.1]{Thomason:1991}. The derived tensor product induces an action of $\perf{X}$ on $\dcat{X}$, as well as on $\sdbcat{X}$. 

\begin{chunk}
A \emph{thick $\otimes$-ideal} $\cat C$ of $\dcat{X}$ is thick subcategory that is closed under the action of $\perf{X}$, that is to say, when $\mcF$ is in $\cat C$ so is $\mcP\lotimes \mcF$ for any perfect complex $\mcP$. The notion of a localizing $\otimes$-ideal is the obvious one.

Given objects $\mcF,\mcG$ in $\dcat{X}$ we say that $\mcF$ \emph{finitely $\otimes$-builds} $\mcG$ if the latter is in every thick $\otimes$-ideal containing the former; and $\mcF$ \emph{$\otimes$-builds} $\mcG$ is each localizing $\otimes$-ideal containing $\mcF$ also contains $\mcG$. For a commutative noetherian ring $R$ and the associated scheme $X\colonequals \spec(R)$, $\otimes$-building coincides with the notion of building discussed earlier, for every thick subcategory of $\dcat{X}$ is thick $\otimes$-ideal.

This leads to the notion of $\otimes$-proxy smallness in $\dcat{X}$. Since this is the only flavor of proxy smallness considered here, we  drop the qualifier ``$\otimes$-", and speak of proxy smallness of objects in $\dcat{X}$.
\end{chunk}

The following result, implicit in the proof of \cite[Lemma~4.1]{Stevenson:2014} often reduces questions of proxy smallness over schemes to the affine case.

\begin{lemma}
\label{le:greg}
Let $X = \cup_{i=1}^n U_i$ be a finite Zariski open cover. Then $\mcF$ in $\dcat{X}$ is proxy small if and only if its restriction $\mcF|_{U_i}$ is proxy small in $\dcat{U_i}$ for each $i$. \qed
\end{lemma}

\begin{chunk}
A quasi-compact scheme $X$ is \emph{locally complete intersection} if there exists a finite open cover $X = \cup_{i=1}^n U_i$, such that $U_i$ is isomorphic to $\spec(R_i)$ where $R_i$ is a locally complete intersection ring, in the sense of \ref{ch:ci-rings}.
\end{chunk}

Given the definition of locally complete intersection schemes and Lemma~\ref{le:greg}, the following result is an immediate consequence of Corollary~\ref{cor:josh}.

\begin{theorem}
\label{th:lci-spa}
Let $X$ be a noetherian separated scheme. Then the following conditions are equivalent
\begin{enumerate}[\rm(1)]
\item $X$ is locally complete intersection;
\item Each object in $\sdbcat{X}$ is proxy small. \qed
\end{enumerate}
\end{theorem}

In the same vein, Theorem~\ref{th:global-eft} readily implies the following global statement.

\begin{theorem}
\label{th:lci-spa-maps}
Let $f\colon Y\to X$ be a flat, essentially of finite type, separated morphism of noetherian schemes, and $\delta \colon Y \to { Y}\times_{X} {Y}$ the diagonal embedding. Then $f$ is locally complete intersection if and only if $\delta_*{\spa O}_Y$ is proxy small in $\dcat{{Y}\times_{X} {Y}}$. \qed
\end{theorem}

It seems likely that using the machinery of derived algebraic geometry one can extend the preceding result to maps of finite flat dimension, with the fiber product replaced by the derived fiber product.

\providecommand{\bysame}{\leavevmode\hbox to3em{\hrulefill}\thinspace}
\providecommand{\MR}{\relax\ifhmode\unskip\space\fi MR }

\providecommand{\MRhref}[2]{
  \href{http://www.ams.org/mathscinet-getitem?mr=#1}{#2}
}
\providecommand{\href}[2]{#2}

\end{document}